%% file: main.tex
\documentclass[12pt]{amsart}
\usepackage{amsmath}
\usepackage{amsfonts}
\usepackage{amssymb}
\usepackage{imakeidx}
\makeindex
\usepackage[colorlinks=true, allcolors=blue]{hyperref}
\usepackage{amsthm}
\usepackage{amscd}
\usepackage[T1]{fontenc}
\usepackage{comment}
\usepackage{appendix, pdfpages}
\usepackage{graphicx}
\usepackage{pdfpages}
\usepackage{subcaption}
\usepackage{yfonts}
\usepackage{epic}
\usepackage{xifthen}
\usepackage{pdfpages}
\usepackage{transparent}
\usepackage{import}
\newtheorem{theorem}{Theorem}[section]
\newtheorem{lemma}[theorem]{Lemma}
\newtheorem{proposition}[theorem]{Proposition}
\newtheorem{corollary}[theorem]{Corollary}
\newtheorem{definition}[theorem]{Definition}
\theoremstyle{remark}
\newtheorem{remark}[theorem]{Remark}

\title[EXOTIC 4-MANIFOLDS WITH $D_\infty$ FUNDAMENTAL GROUP]{EXOTIC STRUCTURES ON 4-MANIFOLDS WITH INFINITE DIHEDRAL FUNDAMENTAL GROUP}
\author{SIMONE TAGLIENTE$^{1,2}$}

\def\Z{\mathbb{Z}}

\DeclareMathOperator{\CP}{\mathbb{CP}}

\begin{document}

\maketitle

\begin{abstract}
The aim of this paper is to produce infinite exotic structures on smooth closed oriented $4-$manifolds with fundamental group isomorphic to the infinite dihedral group, assuming that $b_2^+$ and $b_2^-$ are at least $12$.
\end{abstract}

\section{Introduction}

Constructing exotic structures on closed $4$-manifolds is one of the main goals of modern research in low-dimensional topology.\\
We say that a topological manifold $M$ admits infinitely many exotic structures if there exists a family $\{M_n\}$ of smooth manifolds homeomorphic to $M$, but pairwise non-diffeomorphic. In this paper, we will focus on manifolds with fundamental group isomorphic to the infinite dihedral group, which is the non-trivial semidirect product of $\Z$ and $\Z_2$, and we will use the recent classification result of these manifolds up to homeomorphism \cite{hillman2025homotopy}.\\
From now on, all manifolds are assumed to be smooth.

\begin{theorem}\label{Teo1}
  Fix positive integers $a,b$ with $a,b\geq 12$. Then there exists an infinite family of homeomorphic but pairwise not diffeomorphic closed manifolds with fundamental group isomorphic to the infinite dihedral group, $b_2^+=a$, and $b_2^-=b$.\\
  When $a=b$, the families consist of irreducible manifolds.
\end{theorem}

During the proof, we will also find new examples of exotic manifolds with infinite cyclic fundamental group and signature $0$:
\begin{theorem}\label{exoticS1}
 For $m\geq 13$, there exists a family of manifolds homeomorphic to $(S^1\times S^3) \# (2m) (\CP^2 \# \overline{\CP}^2 )$, but pairwise not diffeomorphic.
\end{theorem}
This will be proven in Theorem \ref{TeoexoticWm}.\\

The infinite dihedral group $D_\infty$ is the only non-trivial semidirect product between $\Z$ and $\Z_2$, with presentations 
\begin{equation*} D_\infty=\langle\sigma,\tau|\sigma^2=1,\sigma\tau\sigma=\tau^{-1}\rangle=\langle\sigma\tau,\tau|\sigma^2=1,(\sigma\tau)^2=1\rangle,\end{equation*}
and so is also equal to the free product $\Z_2 * \Z_2$. Notice that its abelianization is $\Z_2 \oplus \Z_2$. In particular, a manifold having $D_\infty$ as fundamental group has $b_1$ equal to $0$. 

This implies that among the examples of $4$-manifolds that we produce in Theorem \ref{Teo1}, some have vanishing Seiberg-Witten invariant. In fact, a manifold having non-trivial invariants has $1-b_1+b_2^+$ even, and in particular, if $b_1$ is zero, then $b_2^+$ must be odd.\\
Our examples have $b_2^+$ of both parities, which is why they are particularly interesting. We will distinguish their diffeomorphism type using Seiberg-Witten invariants of their connected double covers, see Remark \ref{doublecover}.\\


The first exotic closed $4$-manifold was discovered by Donaldson \cite{donaldson1987irrationality}. Since then, many other examples were produced, including infinitely many exotic manifolds homeomorphic to $\CP^2\#2\overline{\CP}^2$ \cite{akhmedov2010exotic}. These are the smallest simply-connected exotic manifolds known to date.\\
Regarding non-simply-connected manifolds, a vast collection was produced by Torres \cite{torres2011geography,torres2014geography}. In \cite{levine2023new}, the first example of a negative definite exotic manifold with order $2$ fundamental group was given. Later, in \cite{stszbayk}, the collection of manifolds with cyclic fundamental group was extended, including manifolds with $b_2=1$. These are the smallest examples of exotic manifolds up to date.\\

\subsection*{Idea of the proof of Theorem \ref{Teo1}} We will use simply-connected symplectic manifolds as building blocks, and glue them in such a way to obtain a symplectic manifold $W$ with infinite cyclic fundamental group, and endowed with a free involution. The quotient with respect to this involution is a manifold $Y$ with $\pi_1(Y)\cong D_\infty$. We can then create an infinite family of non-diffeomorphic manifolds by knot- or torus-surgery, and we will then verify that the equivariant intersection form is unchanged in this process.\\
However, as we will see, the classification of $4$-manifolds with $D_\infty$ fundamental group up to homeomorphism involves both the equivariant intersection form and the so-called $k$-Postnikov invariant. We can control this invariant under knot-surgery but not under torus surgery, so only in the first case we can conclude that the manifolds are homeomorphic. 

\subsection*{Organization} In Section $2$, we discuss basic operations with symplectic manifolds and we state the homeomorphism classification result that our work is based on. In Section $3$ we describe the building blocks that we will use in the construction. In Section $4$ we discuss the general case and show how to construct exotic manifolds with $D_\infty$ fundamental group starting from a given manifold satisfying some hypothesis and lastly, in Section $5$, we provide examples where the conditions are met.\\

\subsection*{Acknowledgments} I would like to thank Andras Stipsicz and Marco Marengon for their constant support through the project. I would also like to thank Mark Powell for helpful conversations and for his interest in the project. The author was supported by the grant EXCELLENCE-151337.

\section{Basics}
In this section, we recall some general invariants and constructions of $4$-manifolds. 
\subsection{Seiberg-Witten invariant} The Seiberg-Witten invariant of a $4$-manifold $M$ is a function 
$$SW: Spin^c(M)\rightarrow\Z.$$ We can also define the formal sum
$$SW_M:=\sum_{s\in Spin^c(M)} SW(s)e^s.$$
When there is no $2$-torsion in $H_2(M,\Z)$, the map $s \in Spin^c(M) \mapsto c_1(s) \in H^2(M,\Z)$ gives a bijection between $Spin^c(M)$ and the set of characteristic elements of $H^2(M,\Z)$. With a slight abuse of notation, we will identify these two sets (and the set of characteristic elements of $H_2(M,\Z)$, by Poincaré-duality).\\
We are not going to define the invariant explicitly, but we will use some important properties of it, which will allow us to differentiate smooth structures on a given manifold.\\

A closed symplectic manifold with $b_2^+>1$ has non-vanishing Seiberg-Witten invariant \cite{taubes1994seiberg}, and gluing two symplectic manifolds along symplectic submanifolds with trivial normal bundles (operation called 'normal connected sum') produces a new symplectic manifold, if the gluing map is chosen appropriately \cite[Theorem 1.3]{gompf1995new}. Some operations allow us to change the Seiberg-Witten invariant in a controlled manner.
\begin{subsection}{Torus surgery}\label{torussurgery} Consider a torus $T$ in a $4$-manifold $M$ with trivial normal bundle, and let $\lambda$ be any simple closed curve on the torus. Fix a framing on the torus, which induces an identification of the normal bundle with $T\times D^2$, and use it to define a push-off of $\lambda$ on $\partial(\nu T)$. Let $\mu$ denote the meridian of $T$.\\
Fix $p,q$ coprime. We construct the manifold $M_{T,\lambda}(p/q)$ (obtained by torus surgery along $\lambda$, with respect to the chosen framing, and with $\frac{p}{q}$ coefficients) by removing $\nu T$ and gluing back $T\times D^2$, with a diffeomorphism $\phi$ of the boundaries sending $\partial D^2$ to $p \mu + q \lambda$ (the choice of $\phi(\partial D^2)$ is enough to determine the manifold $M_{T,\lambda}(p/q)$). \\
If we fix the base point in $\partial(\nu T)$, we have
\begin{equation}\label{eq}
\pi_1(M_{T,\lambda}(p/q))=\pi_1(M-T)/< \mu^p\lambda^q=1>.
\end{equation}

We can control how the Seiberg-Witten invariant changes under torus surgery. Assume that there is no $2$-torsion in the second homology group of $M$, $M_{T,\lambda}(p/q)$, and of $M_{T,\lambda}(0)$. In \cite{morgan1997product}, it was shown that, given a characteristic class $k\in H_2(M,\Z)$ (see also \cite[Lecture 2]{fintushel2006six})
\begin{equation}\label{EqMMS}
\begin{split}
    \sum_{i}SW_{M_{T,\lambda}(p/q)}(k_{p/q}+2i[T_{p/q}])=p\sum_{i}SW_{M}(k+2i[T])+
    q\sum_{i}SW_{M_{T,\lambda}(0)}(k_{0}+2i[T_0]),
    \end{split}
\end{equation}
where $T_{p/q}$ denotes the core torus of the surgery, and $k_{p/q}$ denotes a fixed class that agrees with $k$ on the complement of the torus $T$ (and similarly for $T_0$ and $k_0$). Moreover, there are no other Spin$^c$ structures on $M_{T,\lambda}(p/q)$ other than $k_{p/q}+2i[T_{p/q}]$ whose restriction on $M_{T,\lambda}(p/q)-\nu T$ agree with $k$. \\
We used a slight abuse of notation in the above formula: the sum is meant to be on the set of Spin$^c$ structures extending $k$ on $M_{T,\lambda}(p/q)-\nu T$ (and similarly for the other terms); in particular, if $T_{p/q}$ is null-homologous, the sum consists of only one term.


This formula was used in the following:
\begin{theorem}\cite[Theorem 1]{fintushel2007reverse}\label{Teoreverse}
Let $M$ be a closed oriented smooth four-manifold, $T$ a null-homologous torus, and fix a framing of its tubular neighborhood $\nu T$. Let $\lambda\in T$ be a simple closed curve
with push-off null-homologous in $M-\nu T$. If the four-manifold $M_{T,\lambda}(0)$ has
non-trivial Seiberg-Witten invariant, in the sense that for some basic class $k_0$, $\sum_i SW_{M_{T,\lambda}(0)}(k_0+2i[T_0])\neq 0$, then the set $\{M_{T,\lambda}(1/n)\}$ contains infinitely many pairwise non-diffeomorphic four-manifolds.
\end{theorem}

\end{subsection}
\subsection{Luttinger surgery}Suppose that the $4$-manifold $M$ admits a symplectic structure, and consider a Lagrangian torus $T$ embedded in $M$. Then, $T$ has a canonical framing (known as the Lagrangian framing), so $\nu T$ is canonically identified with $T \times D^2$. If the parameter of the torus surgery is of the form $\frac{1}{q}$ for some $q$, we call the surgery a 'Luttinger surgery'. Then, the new manifold also admits a symplectic structure \cite{luttinger1995lagrangian}.
\subsection{Knot-surgery}This consists of deleting $\nu T$ and gluing back $(S^3-\nu K)\times S^1$, with $K$ a knot in $S^3$. The only restriction on the gluing diffeomorphism is to send the longitude of the knot to a meridian of $T$. 
Then, the Seiberg-Witten invariant of a given manifold after knot-surgery changes by multiplication by the Alexander polynomial of the knot $K$ \cite{fintushel1998knots}:
\begin{equation}\label{knotsurgeryformula}
    SW_{M_K}=SW_M \cdot\Delta_K(e^{2T}).
\end{equation}

\subsection{Subgroups of $D_\infty$}\label{doublecover} The following Lemma shows that there is a preferred double cover associated to a manifold with fundamental group isomorphic to $D_\infty$:
    \begin{lemma} 
    There is a unique index $2$ subgroup of $D_\infty$ isomorphic to $\Z$.
    \end{lemma}
    \begin{proof}Subgroups of index $2$ of $$D_\infty=\langle\sigma,\tau|\sigma^2=1,\sigma\tau\sigma=\tau^{-1}\rangle$$
    correspond to surjective maps $D_\infty \rightarrow \Z_2$, and so are in bijection with index $2$ subgroups of its abelianization, which is $\Z_2 \oplus \Z_2$. So, there are three subgroups of index $2$ in $D_\infty$. However, only $\langle\tau\rangle$ is isomorphic to $\Z$, while the other two are isomorphic to $D_\infty$ (these are $\langle\tau^2,\sigma\rangle$ and $\langle\tau^2,\sigma\tau\rangle$). \\
    An intuitive way to visualize it is the following: in the dihedral group $D_{2n}$ of isometries of a $2n$-gon, we find a first copy of $D_n$ generated by reflections along lines through opposite vertices, and a second copy generated by reflections across lines through mid-points of opposite edges. 
    \end{proof}
    We can thus differentiate smooth structures on a manifold with $D_\infty$ fundamental group by differentiating smooth structures on the unique double cover corresponding to $\Z \triangleleft D_\infty$.

\subsection{Homeomorphism classification} A crucial point is proving that the manifolds we will construct are homeomorphic. For manifolds with fundamental group isomorphic to $D_\infty$, the topological classification was established recently by  Hillman, Kasprowski, Powell, Ray. Before stating the result, we introduce some notation:
\begin{itemize}
    \item The equivariant intersection form of $M$ is the function $$\lambda_M:\pi_2(M)\times \pi_2(M)\rightarrow \Z[\pi_1(M)]$$ defined as $\lambda_M(\alpha,\beta)=\sum_{g\in \pi_1(M)}(\alpha\cdot g^{-1}\beta)g$. Recall the identification $\pi_2(M)\cong H_2(\Tilde{M},\Z)$, where $\Tilde{M}$ is the universal cover of $M$, and $\pi_1(M)$ acts on $\Tilde{M}$ by deck transformations. So, in the above formula, $\alpha\cdot g^{-1}\beta$ denotes the intersection of (representatives of) $\alpha$ and the translated image of $\beta$ under the action of $g^{-1}$. Also, notice that the sum is finite.
    \item The $n$-Postnikov approximation of $M$ is the complex $P_nM$, together with a map $p_n:M\rightarrow P_nM$, characterized by the property that $p_n$ induces isomorphisms $\pi_i(M) \cong\pi_i(P_nM)$ for $i\leq n$, while $\pi_i(P_nM)=0$ for $i > n$.
    \item $P_2M$ is classified by an element $k_M \in H^3(D_\infty,\pi_2(M))$, called the Postnikov $k$-invariant of $M$.
\end{itemize}

\begin{theorem}\cite[Theorem 1.10]{hillman2025homotopy},\label{ThankyouMark} 
Let $M$, $N$ be closed, oriented, smooth 4-manifolds together with isomorphisms $\pi_1(M)\cong D_\infty, $ $ \pi_1(N)\cong D_\infty$.

Suppose the equivariant intersection forms
$\lambda_M$ and $\lambda_N$
are isometric, via an isomorphism on $\pi_2$ that induces an isomorphism
$$H^3(D_\infty; \pi_2(M)) \rightarrow H^3(D_\infty; \pi_2(N))$$ sending
$k_M \mapsto k_N$.

Then M and N are homotopy equivalent, and if they are either spin or both have universal cover not spin, then M and N are homeomorphic.

\end{theorem}

\section{Building blocks}\label{section 3}
Our main building blocks will be manifolds constructed in \cite{bayku} by Baykur and Hamada. These are the smallest known examples of exotic structures on simply-connected manifolds with signature zero. Here, we will recall the construction and some important properties.\\

\begin{theorem}\cite{bayku}\label{sign0manifolds}
    For each integer $m \geq 4$, there exists an infinite family of pairwise non-diffeomorphic 4-manifolds in the homeomorphism class of $\#_{2m+1}(\CP^2 \#\overline{\CP^2})$.
\end{theorem}

Baykur and Hamada constructed the family in two ways. We will denote these families $\{Z_{m,k}\}$ and $\{Z_{m,J}\}$, where $J$ refers to a knot in $S^3$. We will also denote the manifold $Z_{m,1}$ and $Z_{m,U}$ by $Z_m$.
We will now briefly recall the construction of these manifolds and state some important properties that we are going to use later on.

\begin{figure}[h]
    \def\svgwidth{\columnwidth}
    \centering
    
    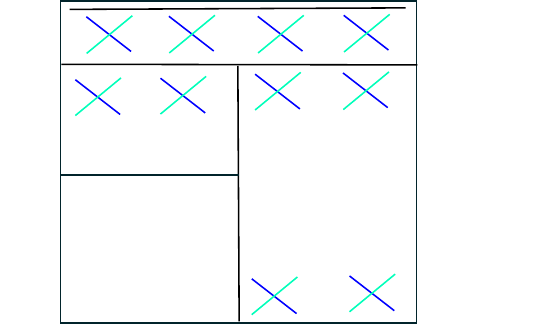

    \caption{The fibration $\textgoth{Z}_m$ and its decomposition into blocks $A$, $B_0$, $C_0$. The red asterisk indicates the chosen base-point (the intersection between the three blocks is diffeomorphic to $S^1$, and we choose the base-point there).\\
    Surgeries on the link of Lagrangian tori displayed in $C_0 \cup B_0$ in dark blue yield the simply connected manifold $Z_m$, with a similar decomposition $Z_m=A\cup B \cup C$.\\
    We also define $\overline{Z}_m$ to be $K\cup B \cup (C - \Sigma)$.}
    \label{fibration}
\end{figure}

\subsection{The fibration $\textgoth{Z}_m$}\label{subsectTheFibrZm}
   Baykur and Hamada constructed a closed Lefschetz fibration $\textgoth{Z}_m$ over the genus $2$ surface, with generic fiber of genus $m+1$.\\
    $\textgoth{Z}_m$ consists of three pieces $A$, $B_0$, $C_0$, as in Figure \ref{fibration}.
    In each of these blocks, the fiber has one boundary component. $A$ is obtained from a non-trivial fibration $K$ over the torus by enlarging the genus of the fiber, while $B_0$ is a trivial fibration (also over the torus), glued to $A$ by a fiber-sum operation. The genus of the fiber of both $A$ and $B_0$ is $m$.\\
    $C_0$ is also trivial, isomorphic to $\Sigma_2\times \mathring{T}_2$. We will use the following:
    \begin{enumerate}
    \item\label{oddity} $K$ contains $4$ nodal fibers, one of which is reducible. The components of the reducible fiber give rise to closed surfaces in $\textgoth{Z}_m$ and $Z_m$ with odd self-intersection.
    \item\label{generatorsofZm} Since $\textgoth{Z}_m$ is a Lefschetz fibration (admitting a section), its fundamental group is generated by loops forming a basis of the section and loops forming a basis of the fiber. Specifically, we can choose a symplectic basis ${x_1, y_1, x_2, y_2}$ of the first homology group of the section, and a symplectic basis ${a_1,b_1, ..., a_{m+1},b_{m+1}}$ of the fiber, and connect everything to a chosen base-point to obtain a set of generators for $\pi_1(\textgoth{Z}_m)$. Here $a_i,b_i$ are dual to each other, and similarly $x_j,y_j$.
     \item\label{vancycles} The curves $a_2$ and $a_1a_2b_2^{-1}a_3[b_3,a_3]b_2$ are vanishing cycles when considered in $K$.
    \end{enumerate}

    \subsection{From $\textgoth{Z}_m$ to $Z_m$ by surgery}\label{relationsinZm} There exists a link of tori in $B_0 \cup C_0$, displayed in Figure \ref{fibration}, such that Luttinger surgeries on this link produce the simply-connected manifold $Z_m$, with a decomposition $Z_m=A\cup B\cup C$. \\
    The simply-connectedness of $Z_m$ is a consequence of the relations introduced in the fundamental group by the surgeries and the presence of the vanishing cycles; more details will be given in Lemma \ref{Pi1}. A key ingredient will be the following (here $\Sigma_g^k$ denotes the genus $g$ surface with $k$ boundary components):
    \begin{proposition}[\cite{bayku}, Proposition 2]\label{BHprop2}
         Consider $X_0=\Sigma^1_g\times \Sigma_1^1$, for $g\geq2$, with the product symplectic structure $\omega_0$. There
are links of embedded Lagrangian tori $L_i$ in $(X_0, \omega_0)$ and Luttinger surgeries along them, which result in symplectic 4-manifolds $(X_i, \omega_i)$, for $i=1,2$, respectively, with the following property:\\
There is a pair of curves $\{a,b\}$ on $\Sigma=\Sigma^1_g\times {pt}$ represented by some $\{\alpha,\beta\}$ in $\pi_1(X_i)$, such
that the quotient of $\pi_1(X_i)$ by the normal closure of $\beta$ is a cyclic group generated by the image of $\alpha$ under the quotient map, where\\
(a) for $i=1$, $a$ and $b$ intersect transversally once in $\Sigma$,\\
(b) for $i=2$, $a$ and $b$ are disjoint and together do not separate $\Sigma$.\\
In particular, $\pi_1(X_i)$ is the normal closure of $\{\alpha, \beta\}$.
    \end{proposition}

     The proof of this proposition is based on a careful analysis of the relations coming from Luttinger surgery. These are of the form 
     \begin{equation}\label{eq1}
         \mu \lambda =1,
     \end{equation}
     where $\mu$ is the meridian of the torus and $\lambda$ the surgery slope.

    \subsection{Properties of the manifold $Z_m$}\label{subsectManifldZm} The manifold $Z_m$ is homeomorphic to $\#_{2m+1}(\CP^2 \#\overline{\CP}^2)$: in fact, it has signature $0$ and odd intersection form by Subsection \ref{subsectTheFibrZm} (\ref{oddity}), and being simply connected, this is enough to pin down its homeomorphism type. However, since $Z_m$ is obtained by Luttinger surgeries from a symplectic manifold, it admits a symplectic structure, and, in particular, its Seiberg-Witten invariant is non-vanishing. We deduce that $Z_m$ is homeomorphic but not diffeomorphic to $\#_{2m+1}(\CP^2 \#\overline{\CP}^2)$.\\
     We are particularly interested in some special submanifolds of $Z_m$:
    \begin{enumerate}
        \item\label{Sigma in Zm} There exists an embedded genus $2$ surface $\Sigma$ in $Z_m$ (descending from a section of $\textgoth{Z}_m$) of $0$ self-intersection and simply connected complement;
        \item\label{newtori}  Each time we enlarge the genus of the fiber of $\textgoth{Z}_m$ by one, we produce two new pairs of Lagrangian dual tori in the block $A$ (given by $x_1 \times a_l$, $x_1 \times b_l$, $y_1 \times a_l$, $y_1 \times b_l$, for $5\leq l \leq m$; so for example $x_1 \times a_l$ is dual to  $y_1 \times b_l$), and so they descend to Lagrangian tori in $Z_m$. As explained in (\cite{bayku}, Addendum 11), each one of these new tori has simply connected complement in $Z_m$ and is available for surgery procedures.
    \end{enumerate}
    \subsection{Exotica}\label{subhowgetexotica} 
    There are two ways to produce exotic manifolds associated to $Z_m$:
\begin{definition}\label{def1}
    Let $Z_m$ be the manifold described in Figure
    \ref{fibration}, and let $T$ be the core torus of the surgery on $ x_2\times a_{m+1}$, with the notation of Section \ref{subsectTheFibrZm} (\ref{generatorsofZm}). We define the family $\{Z_{m,k}\}$ as the family of manifolds obtained by torus surgery with parameter $\frac{1}{k}$ on $T$.
\end{definition}

In the above definition, $T$ is null-homologous because surgery on $x_2\times a_{m+1}$ kills a loop in the fundamental group, so it decreases $b_1$ by one and $b_2$ by two. 
\begin{definition}\label{def2}
    Let $Z_m$ be the manifold described in Figure
    \ref{fibration}. Let $T\subset A-K$ be a Lagrangian torus obtained by enlarging the genus of the fiber, as described in Section \ref{subsectManifldZm} (\ref{newtori}). We define the family $\{Z_{m,J_k}\}$ as the manifolds obtained by knot surgery on $T$ using a family of knots $\{J_k\}$ with distinct Alexander polynomials. 
\end{definition}

In each of these families there is an infinite subfamily of exotic copies of $\#_{2m+1}(\CP^2 \#\overline{\CP}^2)$. In the first case, we distinguish the diffeomorphism type by Theorem \ref{Teoreverse}, while in the second, we use the knot-surgery formula \eqref{knotsurgeryformula}. For notational purposes, we will denote the second family simply by $\{Z_{m,J}\}$.

\subsection{Other Lemmas} Fundamental group computations require extra care. We will need the following lemmas. Firstly, we will strengthen Proposition \ref{BHprop2} (recall that, when talking about normal closures of loops in the fundamental group, we can ignore the choice of the base-point):
\begin{corollary}\label{corBH}
     Recall that $\Sigma_g^n$ denotes the n-punctured genus-g surface. Consider $X_0=\Sigma^1_g\times \Sigma_1^2$, for $g\geq2$, with the product symplectic structure $\omega_0$. Let $\delta$ be one of the two boundary curves of $\Sigma_1^2$. For any pair of curves $\{a,b\}$ on $\Sigma=\Sigma^1_g\times {pt}$, that are either dual or disjoint and linearly independent, there are links of embedded Lagrangian tori $L_i$ in $(X_0, \omega_0)$ and Luttinger surgeries along them, which result in symplectic 4-manifolds $(X, \omega)$, with the following property:\\
The quotient of $\pi_1(X)$ by the normal closure of $\{b, \delta\}$ is a cyclic group generated by the image of $a$ under the quotient map.\\
In particular, $\pi_1(X)$ is the normal closure of $\{a, b, \delta\}$.
\end{corollary}

\begin{proof}
    Notice that we can obtain $\Sigma_1^2$ from $\Sigma_1^1$ by attaching a pair of pants $P$. Let us fix a base point on the boundary component $\delta_0$ of $P$ that we glue to $\Sigma_1^1$,
    so that $\pi_1(P)$ is freely generated by $\delta_0$ and $\delta$ (here we have to connect $\delta$ to the base-point).
    So, by Van-Kampen Theorem, $$\pi_1(\Sigma_1^2)=\pi_1(\Sigma_1^1)*\langle\delta\rangle.$$
    We then see that $\pi_1(\Sigma^1_g\times \Sigma_1^2)=\pi_1(\Sigma^1_g) \times (\pi_1(\Sigma_1^1)*\langle\delta\rangle)$. This means that if we quotient this group by the normal closure of $\delta$, we obtain $\pi_1(\Sigma^1_g\times \Sigma_1^1)$.\\
   We now apply Proposition \ref{BHprop2}: if we perform all the prescribed surgeries on the same link of tori in $\Sigma^1_g\times \Sigma_1^2\subset \Sigma^1_g\times \Sigma_1^1$, we see that the fundamental group is normally generated by the curves $a, b,\delta$. Moreover, if we quotient by the normal closure of $\delta$ and $b$, we obtain a cyclic group generated by $a$.
\end{proof}

We can now prove the following:
\begin{lemma}\label{Pi1}
    Recall that $Z_m$ can be expressed as $A\cup B\cup C$ as in Figure \ref{fibration}, that $K\subset A$, and that $\Sigma \subset C$ descends from a section (for their precise definitions, see Section \ref{relationsinZm} and \ref{subsectManifldZm} (\ref{Sigma in Zm})). Then, $\pi_1(K\cup B \cup (C-\Sigma))=1$.
\end{lemma}
\begin{proof}
    Notice that $C_0-\Sigma=\Sigma^1_2\times \Sigma_1^2$. We will follow the arguments used in \cite{bayku}, but with Corollary \ref{corBH} instead of Proposition \ref{BHprop2}. Choose $\delta$ as the boundary curve on $\Sigma_1^2$ that is glued to $\Sigma^1_g$.\\
    Then, as done in \cite{bayku}, we perform Luttinger surgeries (with coefficient equal to $1$) on a link of tori in $C_0$ and $B_0$, and call $C$ and $B$ the results after the surgeries. By Corollary \ref{corBH}, we can arrange that $\pi_1(C-\Sigma)$ is normally generated by $x_2, y_2$, and $\delta$. By Proposition \ref{BHprop2}, we can arrange that $\pi_1(B)$ is normally generated by $a_2$ and $a_1a_3$, and that the quotient of $\pi_1(B)$ by the normal closure of $a_2$ is free, generated by $a_1a_3$.
    We can compute the fundamental group of $\overline{Z}_m:=K\cup B\cup (C-\Sigma)$ by Van Kampen Theorem (twice): when we glue $B$ and $C$, the curves $x_2,y_2,$ and $\delta$ are included in the normal closure of $\{a_2,a_1a_3\}$. Moreover, as seen in Section \ref{subsectTheFibrZm}(\ref{vancycles}), $a_2$ is a vanishing cycle, so, when we include this relation in $K$, we obtain that the fundamental group of $\overline{Z}_m$ is cyclic (hence abelian) generated by $a_1a_3$. However, the abelianization of the relation introduced by the second vanishing cycle is exactly $a_1a_3=1$, and so we deduce that $\overline{Z}_m$ is simply connected.
\end{proof}

\begin{corollary}\label{cor1}
    Assume the notation of Lemma \ref{Pi1}, and let $a_i,b_i,x_j,y_j$ be as in Section \ref{subsectTheFibrZm} (\ref{generatorsofZm}). Let $T_1=x_1\times a_5$, $T_2=y_1\times a_5'$, with $a_5'$ a disjoint translated copy of $a_5$, and $T_3=x_1\times a_6$ be Lagrangian tori in $A-K \subset Z_m$. Then,
    $\pi_1(Z_{m}-(\Sigma \cup T_1 \cup T_2))=1$ for $m\geq5$, while
    $\pi_1(Z_{m}-(T_1 \cup T_2\cup T_3))=1$ for $m \geq 6$.
\end{corollary}
\begin{proof}
    Notice that we need the assumptions on $m$ to ensure that we are able to find the tori $T_1$, $T_2$, $T_3$ in $A-K$, each one with a dual torus, as described in Section \ref{subsectManifldZm}(\ref{newtori}). Notice that also $\Sigma$ admits a dual surface, descending from the fiber in $\textgoth{Z}_m$.\\
    The key point is that, when removing a torus or the surface $\Sigma$, we introduce meridians of the surfaces as potential new generators in the fundamental group, but these turn out to be null-homotopic. This is because of the presence of dual surfaces: for example, since $\Sigma$ is dual to the fiber of $\textgoth{Z}_m$, the meridian of $\Sigma$ can be expressed as a product of commutators $[a_1,b_1][a_2,b_2]\dots[a_{m+1},b_{m+1}]$. However, since $\pi_1(Z_m-\Sigma)=1$, each of these elements is null-homotopic in the complement of $\Sigma$, and thus so is the meridian.
    \\ So, the fundamental groups of $Z_{m}-(\Sigma \cup T_1 \cup T_2)$ and of $Z_{m}-(T_1 \cup T_2 \cup T_3)$ are still generated by $x_j, y_j, a_i, b_i$.
By Lemma \ref{Pi1}, these cycles are null-homotopic in $\overline{Z}_m=K\cup B_0 \cup (C_0 - \Sigma)$, and, since $\overline{Z}_m\subset Z_{m}-(\Sigma \cup T_1 \cup T_2)$ (and similarly for $Z_{m}-(T_1 \cup T_2\cup T_3)$), we conclude.
\end{proof}

By the above Corollary, since the tori $T_1$, $T_2$, and $T_3$ are disjoint, we can perform knot surgery on one of them to get exotic copies $Z_{m,J}$ of $Z_m$ (see section \ref{subhowgetexotica}), and the other two tori will survive in $Z_{m,J}$.\\
Remember also that we get the exotic copies $Z_{m,k}$ by torus surgery on a torus which is disjoint from the tori $T_1$ and $T_2$ of Corollary \ref{cor1}, see Section \ref{subhowgetexotica}.

\begin{corollary}\label{cor2}
  Assume the notation of Lemma \ref{Pi1} and Corollary \ref{cor1}, and let $T_1=x_1\times a_5$, $T_2=y_1\times a_5'$. Let $Z_{m,k}$ and $Z_{m,J}$ be as in Section \ref{subhowgetexotica}. Then, for $m\geq5$, $\pi_1(Z_{m,k}-(\Sigma \cup T_1))=1$ and
    $\pi_1(Z_{m,k}-(T_1 \cup T_2))=1$.\\
    Moreover, $\pi_1(Z_{m,J}-(\Sigma \cup T_1))=1$ for $m\geq5$, while
    $\pi_1(Z_{m,J}-(T_1 \cup T_2))=1$ for $m \geq 6$.
    \end{corollary}
\begin{proof}
For what concerns $Z_{m,k}$, we use the same argument as in Lemma \ref{Pi1}. The only difference is a non-Luttinger surgery on a torus in $\overline{Z}_m$: this means that one of the relations introduced in the fundamental group becomes of the form $(\mu^k) \lambda=1$, instead of the one in Equation \eqref{eq1}. In fact, $k$ is both the surgery coefficient (understood as $\frac{k}{1}$) and the parameter of the family. However, the argument goes through in the same exact way. This is standard, used by Baykur and Hamada in their proof and also used many times before, see for example \cite{fintushel2007reverse}.
 (The key point is that the meridian $\mu$ is null-homotopic, and this is due to the other relations in the fundamental group. So, the relation introduced by non-Luttinger surgery is equivalent to the one introduced by Luttinger surgery).\\

Regarding $Z_{m,J}$, using the hypothesis on $m$, we can do knot surgery on $T_2$ when deleting $\Sigma$ and $T_1$ (resp. on $T_3=x_1\times a_6$ when deleting $T_1$ and $T_2$). Then, we use the fact that $\pi_1(Z_{m}-(\Sigma \cup T_1 \cup T_2))=1$ (resp. $\pi_1(Z_{m}-(T_1 \cup T_2\cup T_3))=1$), proven in Corollary \ref{cor1}, to deduce that, after knot-surgery along $T_2$ (resp. $T_3$), we still get simply-connected manifolds.
\end{proof}

\section{A more general Theorem}
In this section we will prove the following Theorem, from which we will deduce Theorem \ref{Teo1}.
\begin{theorem}\label{Teo4}
    Consider a simply connected symplectic manifold $Z$ with $b_2^+(Z)>1$ and with disjointly embedded symplectic connected surfaces $S_1$ and $S_2$, both of zero self-intersection.
    Suppose that you are in one of the two following cases:
    \begin{enumerate}
        \item\label{Hyp1} There exists a null-homologous embedded torus $T$ in $Z-(S_1\cup S_2)$, which contains a curve $\lambda$ null-homologous in $Z-T$. Assume that the manifold $Z_{T,\lambda}(0)$ is symplectic,
        and that the core torus of the surgery $T_0$ admits a dual torus with zero self-intersection. 
        If $\{Z_{k}\}$ denotes the family of exotic manifolds homeomorphic to $Z$, with $Z_1=Z$, obtained as explained in Theorem \ref{Teoreverse}, assume that $\pi_1(Z_k-(S_1 \cup S_2))=1$. 
        \item\label{Hyp2} There exists a torus $T$ embedded in $Z-(S_1 \cup S_2)$, with $T\cdot T=0$, and with a dual surface $T^*$. 
        \\Assume that, for each $\gamma$ simple curve in $S_1$ or $S_2$, there exist a surface $F_\gamma\subset Z -(T\cup T^*)$ with $\partial F_\gamma=\gamma$, that intersect $S_1$ and $S_2$ only on the boundary. \\Lastly, if $Z_J$ denotes the knot surgery on $T\subset Z$ with knot $J$, assume that
        $\pi_1(Z_J-(S_1 \cup S_2))=1$.
    \end{enumerate} 
    Then there exists an infinite family $\{Y_{k}\}$ (respectively $\{Y_J\}$) of pairwise not diffeomorphic manifolds associated to $Z$ with fundamental group isomorphic to $D_\infty$, and with the same equivariant intersection form. In case $(2)$, the manifolds are homeomorphic but pairwise not diffeomorphic.
\end{theorem}

    \begin{remark}
        In Case $(1)$ of Theorem \ref{Teo4}, we can construct the family of exotic manifolds $\{Z_k\}$ using Theorem \ref{Teoreverse} because the assumption on the non-triviality of the Seiberg-Witten invariants is satisfied. This is due to the presence of the dual torus of $T_0$; see \cite[Corollary 2]{fintushel2007reverse}.
    \end{remark}

The proof will go through the following steps (Figure \ref{schematic} describes schematically the manifolds involved):

\begin{enumerate}
    \item Consider the manifold $Z$ and another copy $Z'$, and remove a tubular neighborhoods of $S_1$ and its copy $S_1'$. Then glue the boundaries with the map displayed in Figure \ref{involution} (in the $g(S_1)=2$ case). This produces a simply connected manifold $X$ with a free orientation-preserving involution.
    \item Remove neighborhoods of $S_2$ and $S_2'$, and glue the boundaries similarly. We obtain a manifold $W$ with fundamental group isomorphic to $\Z$, and still endowed with a free involution.
    \item The quotient with respect to such an involution has fundamental group isomorphic to $D_\infty$.
    \item Torus surgery and knot-surgery change the diffeomorphism type but not the equivariant intersection form, and in the case of knot surgery we will obtain infinitely many exotic structures on our manifolds.
\end{enumerate}

\begin{figure}
   \def\svgwidth{0,6\columnwidth}
    \centering
    
    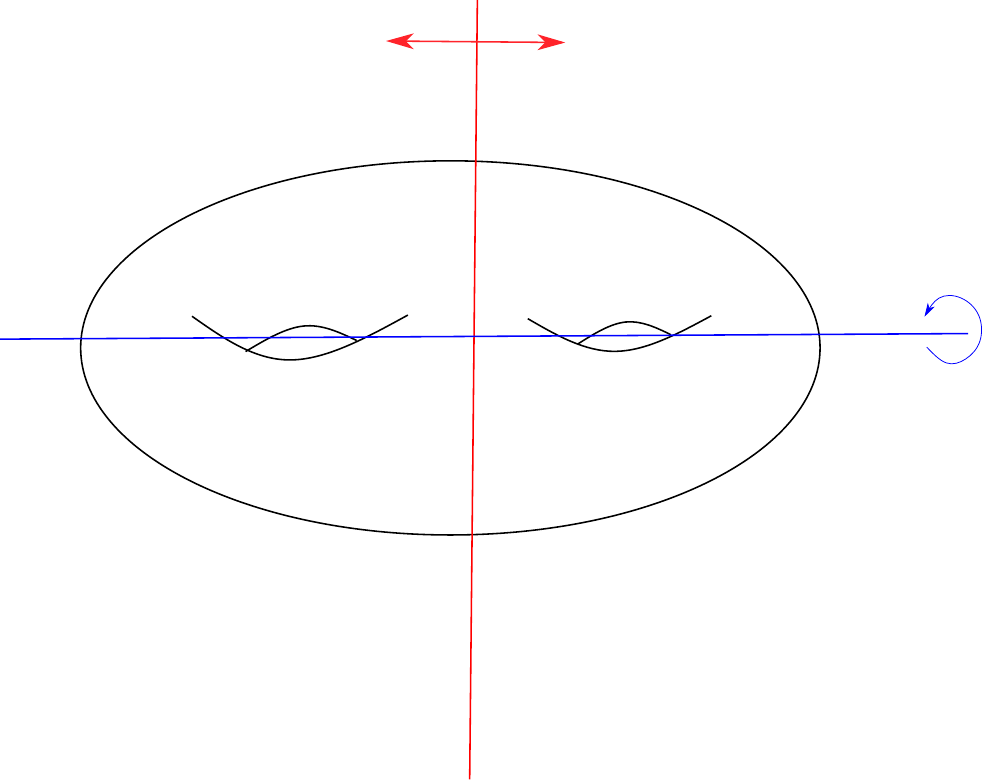

    \caption{When gluing two manifolds with boundary parametrized by $S\times S^1$, where $S$ is a surface, we will always use the map given by the composition between the reflection along the red plane and the $\pi$-degree rotation along the blue axis on $S$, and the identity on $S^1$.}
    \label{involution}
\end{figure}

\begin{proof}[Proof of Theorem \ref{Teo4}]
\begin{figure}[h]
    \def\svgwidth{\columnwidth}
    \centering
    
    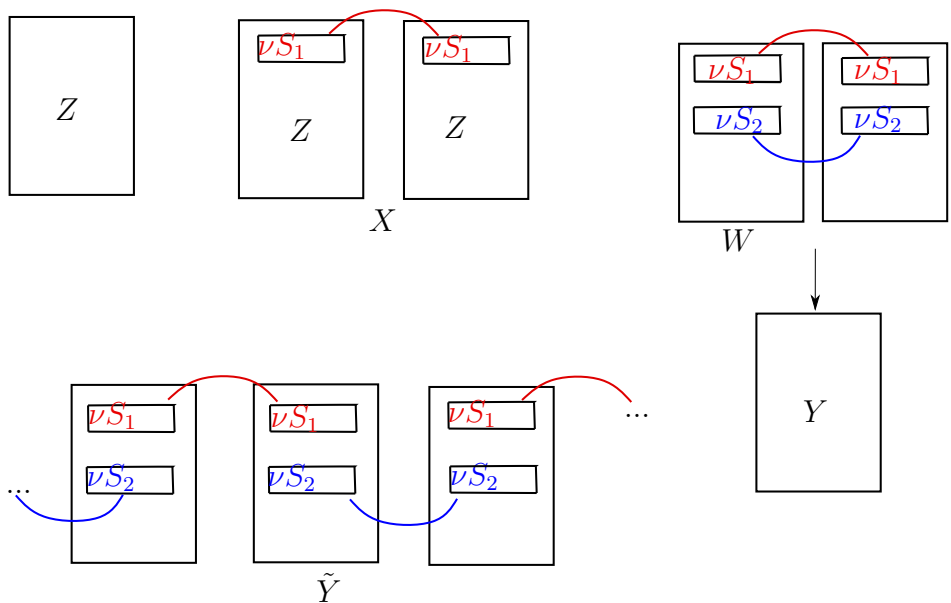

    \caption{A schematic representation of the manifolds $Z$, $X$, $W$, $Y$, and $\tilde{Y}$.}
\label{schematic}
\end{figure}
\textbf{Step (1)}. 

Consider two copies of $Z$, and remove from each one a neighborhood of $S_1$. We glue the boundary components using the map displayed in Figure \ref{involution}, to get an oriented manifold denoted by $X$ (so, $X$ is obtained by performing normal connected sum on two copies of $Z$ across the symplectic submanifold $S_1$). Then, $X$ is simply connected (this easily follows from Van-Kampen Theorem and the fact that $Z-\nu S_1$ is simply-connected) and admits an orientation-preserving free involution $\sigma$, given by the identification of the two $Z$ blocks. The idea of this construction is motivated by (\cite[section $2.2$]{stszbayk}). Note that the involution is free because of the choice of the gluing map, which is a free involution.\\

\textbf{Step (2)}.
Now, remove neighborhoods of $S_2$ and $S_2'$, and we get a simply connected manifold with two diffeomorphic boundary components (simply-connectedness follows from the hypothesis on the complement of the two surfaces). When gluing these (again using the map in Figure \ref{involution}), we will obtain a manifold $W$ with fundamental group isomorphic to $\Z$:

\begin{lemma}\label{cyclic-pi1}
    Let $X$ be a simply connected manifold with two boundary components $M_1$ and $M_2$, and let $\phi: M_1 \rightarrow M_2$ be an orientation-reversing diffeomorphism, where each $M_i$ inherits an orientation from $X$. Then gluing the boundary components via $\phi$ produces a closed oriented manifold $W$ with infinite cyclic fundamental group. A generator is a loop $s$ that lifts to a curve $\Tilde{s}$ connecting a point in $M_1$ to a point in $M_2$.
\end{lemma}

\begin{proof}
    $M_1$ and $M_2$ are identified in $W$, so we call their image $M\subset W$. We use Van-Kampen Theorem on a piece $A \subset W$, given by the union of a neighborhood of $s$ and a closed collar of $M$, and a piece $B$, given by the closure of the complement of the collar. Then $B$ is simply connected, while $A$ is homotopy equivalent to $M \vee S^1$. In particular, $\pi_1(A) \cong \Z * \pi_1(M)$. The intersection $A \cap B$ is given by two push-offs of $M$ connected by (a sub-path of) $s$, and its fundamental group will kill the contribution from $\pi_1(M)$ in $\pi_1(A)*\pi_1(B)$.
\end{proof}
This lemma also tells us that a generator of $\pi_1(W)$ is a loop connecting $S_1$ and $S_2$ in the first block and $S_1'$ and $S_2'$ in the second.\\

The involution $\sigma$ described before extends after the gluing, so $W$ admits also an orientation-preserving free involution, which we still denote $\sigma$.


\begin{remark}
    The manifold $W$ admits a symplectic structure, since it is built from $Z$, which admits a symplectic structure $\omega$, by performing normal connected sums along symplectic submanifolds. Notice that the gluing map must be induced by a symplectomorphism of the surfaces $S_1$ and $S_1'$, as prescribed in \cite{gompf1995new}. Now, the map shown in Figure \ref{involution} is a diffeomorphism between $S_1$ and $-S_1'$ (since it reverses the orientation), and so, in order for $-S_1'$ to be a symplectic surface in $Z'$, we have to endow $Z'$ with the symplectic structure $-\omega$.\\
\end{remark}

Since $W$ is symplectic, the Seiberg-Witten invariant of $W$ is non-vanishing when $b_2^+(W)>1$. Notice that, by construction, and since $b_1(W)=1$, we have 
$$\chi(W)=b_2^+(W)+b_2^-(W) \geq 2\chi(Z)=2(2+b_2^+(Z)+b_2^-(Z)),$$

(in the inequality we used that $S_1$ and $S_2$ cannot be spheres, since they have zero self-intersection and $b_2^+(Z)>1$) and from $$\sigma(W)=b_2^+(W)-b_2^-(W)=2\sigma(Z)=2(b_2^+(Z)-b_2^-(Z)),$$ one easily deduces that $b_2^+(W)>1$ always hold.
\bigskip

\textbf{Step (3)}. Let $Y$ denote the quotient of $W$ by the involution $\sigma$. The following shows that its fundamental group is isomorphic to the infinite dihedral group.

\begin{lemma}\label{Lemma4.2}
    The manifold $Y=W/\sigma$ has fundamental group isomorphic to the infinite dihedral group.
\end{lemma}
\begin{proof}
    It is enough to show a free, orientation-preserving $D_\infty$-action on the universal cover of $W$, such that the quotient with respect to this action is the manifold $Y$.

Figure \ref{univcover} is a schematic picture of it: the universal cover of $W$ is given by infinitely many pairs of blocks $Z$, and we perform normal connected sum along the surfaces in each block as indicated by the red and blue lines. 
Notice that at each step we take two disjoint simply connected manifolds and we glue their connected boundaries, so the result will be simply connected. \\
We have a free $\Z$ action given by translation of two blocks (we denote this translation by $\tau$), and a free $\Z_2$ action that is an extension of the involution $\sigma$ considered in Step $(1)$ (of course, the fact that it is free depends on the choice of the gluing map of the surfaces).\\
We then see that the relation $\sigma \tau \sigma= \tau^{-1}$ holds, and so the group generated by these two actions is $D_\infty$.
\end{proof}

\begin{figure}[h]
    \def\svgwidth{\columnwidth}
    \centering
    
    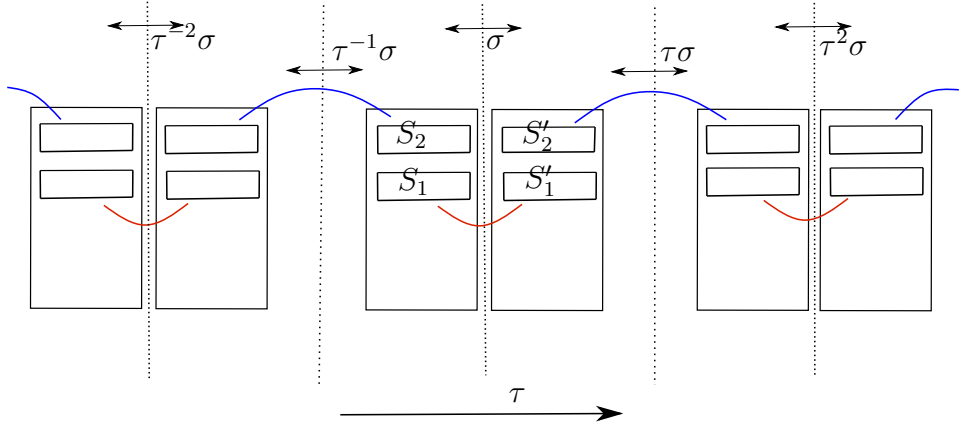

    \caption{The universal cover of $Y$. We remove neighborhoods of surfaces in each block, and identify the boundaries according to the red and blue lines. Each group element sends a copy of $Z-(\nu S_1\cup \nu S_2)$ to another one via the identity map, and permutes them according to the $D_\infty$ action.  So, order $2$ elements act as 'reflections' across the dotted lines on the schematic, while the $\Z$ subgroup acts via translations.}
    \label{univcover}
\end{figure}

\begin{remark}\label{W/sigma}
    In order to construct the manifold $W$ endowed with the involution $\sigma$, such that $\pi_1(W/\sigma)=D_\infty$, we used only the fact that $\pi_1(Z-(S_1 \cup S_2))=1$. Later on, we will run the same construction with the exotic copies $Z_k$ and $Z_J$ of $Z$, and both of them, by assumption, satisfy the same hypothesis on the complement of the surfaces. So, in the same way, we obtain manifolds $Y_k$ and $Y_J$ associated to them with fundamental group isomorphic to $D_\infty$.
\end{remark}

\textbf{Step (4)}. In this section, we will discuss how to produce the family of pairwise not diffeomorphic manifolds with the same equivariant intersection form as $Y$. As stated in Theorem \ref{Teo4}, we will distinguish two cases, based on the type of surgery used to produce exotic copies of $Z$.

\textbf{Torus surgery (using hypothesis (\ref{Hyp1}) of Theorem \ref{Teo4})}. 
Suppose that we can produce exotic copies $Z_k$ of $Z$ by torus surgery on a null-homologous torus $T \subset Z-(S_1\cup S_2)$, and that the manifold $Z_{T,\lambda}(0)$ has non-trivial Seiberg-Witten invariant for some curve $\lambda$. Then, we can perform the same torus surgeries (equivariantly) in each of the two blocks of the double cover $W$ of $Y$ (recall that $W$ has non-vanishing Seiberg-Witten invariant). Denote $W_k$ the double covers after the surgeries, and note that these still have infinite cyclic fundamental group, again by the hypotheses of Theorem \ref{Teo4} and by Lemma \ref{cyclic-pi1}.\\ Each $W_k$ has an involution, and the quotients are the manifolds $Y_k$ that have fundamental group isomorphic to $D_\infty$ by Lemma \ref{Lemma4.2} and Remark \ref{W/sigma}. \\

We now argue that this process changes the Seiberg-Witten invariants of $W_k$. Note that $W_k$ is obtained by $1/k$-surgery on the null-homologous torus $T$ in $Z\#_{S_1,S_2} Z_k$, and by two torus surgeries in $W=Z \#_{S_1,S_2} Z$.\\
 Let $s$ be the basic class associated to the symplectic structure on $W$.
Then, applying Equation \eqref{EqMMS}, 
$$SW_{W_k}(s_{1/k})=SW_{Z\#_{S_1,S_2} Z_k}(s)+
    k\sum_{i}SW_{Z_{T,\lambda}(0)\#_{S_1,S_2} Z_k}(s_{0}+2i[T_0]).$$

Note that in the formula we used the fact that $T$ is null-homologous, as well as $T_{1/k}$. We can simplify this formula even more: by the hypothesis of Theorem \ref{Teo4}, $T_0$ admits a dual torus with zero self-intersection, so we can use the proof of \cite[Corollary 2]{fintushel2007reverse} to deduce that
in the family $\{s_{0}+2i[T_0]\}$ there is at most one basic class. If we use Equation \eqref{EqMMS} once again, we get (we drop $T,\lambda$, and $S_1, S_2$ from the notation)

$$SW_{W_k}(s_{1/k})=SW_{Z\# Z}(s)+
    k SW_{Z(0)\# Z}(s_{0})+  k SW_{Z\# Z(0)}(s_{0}')+ k^2 SW_{Z(0)\# Z(0)}(s_{0,0}),$$
    and the leading term $SW_{Z(0)\#_{S_1,S_2} Z(0)}(s_{0,0})$ is not zero because $Z(0)\#_{S_1,S_2} Z(0)$ is symplectic (by hypothesis, $Z(0)$ is symplectic, and $s_{0,0}$ is the canonical class). This is enough to conclude that the family $\{W_k\}$ contains infinitely many non-diffeomorphic manifolds. \\



It follows that the manifolds $Y_k$ will be non-diffeomorphic, distinguished by their double covers $W_k$ (by Lemma \ref{Lemma4.2}, $\pi_1(W_k)=\Z$, so $W_k$ is the double cover associated to the unique  $\Z \triangleleft D_\infty$ of index $2$, see Section \ref{doublecover}). We will now check that the equivariant intersection form remains unchanged.\\

We have already described the universal cover of $Y$ in Figure \ref{univcover}.
Notice that the effect of non-Luttinger torus surgery on $Y$ (to obtain $Y_k$) on the universal cover is to change each $Z$ block with a $Z_{k}$ one.

It is clear that this operation does not change the equivariant intersection form of $Y$, as long as the null-homologous tori where we perform the surgeries do not intersect other representatives of a basis of $H_2(\Tilde{Y})$. We can always obtain this by a surgery procedure.
Denote by $T$ the null-homologous torus, and by $\Sigma$ any other surface. We have $\Sigma\cdot T=0$ algebraically, and we want to realize this geometrically. Given a pair of positive and negative intersection points $p_+$ and $p_-$, connect them with a path in $T$, disjoint from $\Sigma$ except at its end. We can cut small disks containing $p_+$ and $p_-$, and glue a cylinder to $\Sigma$ following the path. This process deletes the intersection points in pairs, and the result is an orientable manifold in the same homology class as $\Sigma$.\\ 

This concludes the proof of the first part of Theorem \ref{Teo4}.

\begin{remark}
    We proved that the manifolds $Y_k$ constructed via torus surgery have the same equivariant intersection form, however, we are not able to say that the Postnikov $k$-invariant is preserved. That would imply that the manifolds $Y_k$ are homeomorphic but pairwise not diffeomorphic. One wouldn't need to take the $k$-invariant under consideration if the equivariant intersection form is of the form $H(I\pi)\oplus \lambda$, as explained in \cite[Corollary 1.12]{hillman2025homotopy}. It is worth noting that $\pi_2$ of a manifold with $D_\infty$ fundamental group is always stably isomorphic to the direct sum of $I\pi$ and a free $\pi_1$-module, see \cite[Theorem 4.2]{hambleton2009intersection}  (notice that $H_4(D_\infty;\Z)=0$). 
\end{remark}

\textbf{Knot-surgery (using Hypothesis \ref{Hyp2} in Theorem \ref{Teo4})} Here we suppose that there is an embedded torus $T \subset Z-(S_1\cup S_2)$, such that $\pi_1(Z-(S_1\cup S_2\cup T))=1$. Notice that the hypothesis $\pi_1(Z-T)=1$ implies the existence of a surface $T^*$ dual to $T$. By dual surfaces, we mean that they intersect in one point and are disjoint from other elements of a basis of $H_2(Z)$. Recall that, by hypothesis, $T$ has trivial self-intersection.\\
We can thus perform knot-surgery on $T$ to change the Seiberg-Witten invariant of the manifold $Z$, and obtain a family of exotic manifolds $Z_J$, where $J$ denotes the knot used in the surgery. Notice that $Z_J$ is still simply connected. Using a construction similar to the torus surgery case, we obtain manifolds $W_J$ with distinct Seiberg-Witten invariants and infinite cyclic fundamental group, and endowed with an involution. Quotienting by this involution yields manifolds $Y_J$ with $D_\infty$ fundamental group.\\
\begin{remark}\label{symplecticality}
    If knot $J$ used in the surgery is fibered and the torus $T$ is symplectically embedded, then $W_J$ also admits a symplectic structure \cite{fintushel1998knots}.
\end{remark}

The effect of the surgery on the homology of $Z$ is to replace $T^*$ with another surface, obtained by puncturing $T^*$ and gluing to it a Seifert surface of the knot $J$ used in the surgery. Denote this new surface by $T^*_J$. Also, note that a translated copy of $T$ survives in $Z_J$, and we still denote it $T$. Then, $T$ and $T^*_J$ are dual in $Z_J$. Moreover, the effect on the universal cover of $Y$ is to replace $T^*$ with $T^*_J$ in each block.\\

To verify that this process does not change the equivariant intersection form, we will argue that, when gluing the blocks to each other via normal connected sum along $S_1$ or $S_2$, we only create surfaces in homology that still do not intersect $T$ and $T^*$, so that they are still dual in the universal cover of $Y$.\\
Denote $Z^k_{S_1,S_2}:=Z\#_S\dots \#_S Z$ the manifold obtained by gluing $k$ blocks along $S_1$ and $S_2$ alternatively (as prescribed by the arcs in Figure \ref{univcover}), where $S$ denotes either one of the two surfaces. Consider the Mayer-Vietoris sequence for the splitting $Z^{k+1}_{S_1,S_2}=Z^k_{S_1,S_2}\cup_S Z$:
\begin{equation}\label{eqfg}
0\rightarrow H_2(S\times \mu) \rightarrow H_2(Z^k_{S_1,S_2}-\nu(S))\oplus H_2(Z-\nu(S))\xrightarrow{f}
\end{equation}
\begin{equation*}
\rightarrow H_2(Z^k_{S_1,S_2}\#_{S}Z)\xrightarrow{g} H_1(S\times\mu)\rightarrow 0,
\end{equation*}

where $\mu$ is the meridian of the surface $S$ across which we glue the two blocks. Here we used the fact that $H_3(Z^k_{S_1,S_2}\#_{S}Z)=0$ (which follows from $H_1(Z^k_{S_1,S_2}\#_{S}Z)=0$) and $H_1(Z^k_{S_1,S_2}-\nu S)=0$. Remember also that, by hypothesis, $[S]^2=0$, so that $\nu(S)$ is identified with $S\times D^2$.\\
Again by a Mayer-Vietoris argument, we can see that 
$$H_2(Z^k_{S_1,S_2}-\nu(S))=(H_2(Z^k_{S_1,S_2})/\langle S^*\rangle)\bigoplus_i (\gamma_i \times \mu),$$

where $\gamma_i$ are generators of $H_1(S)$.

From Equation \eqref{eqfg}, we see that, when gluing the $(k+1)$-st $Z$ block to $Z^k_{S_1,S_2}$, we introduce new elements in homology, specifically:
\begin{itemize}
    \item Surfaces of the form $\gamma_i\times \mu$, where $\gamma_i$ are generators of $H_1(S)$ (coming from the map $f$ in \eqref{eqfg}),
    \item surfaces of the form $F_{\gamma_i} \cup_{\gamma_i} F'_{\gamma_i}$, where $\gamma_i$ are generators of $H_1(S)$ and $\partial F_{\gamma_i}=\gamma_i$ (these surfaces are sent to generators of $H_1(S\times\mu)$ by the map $g$ in \eqref{eqfg}). We can choose $F_{\gamma_i}$ to be disjoint from the knot-surgery torus and its dual, by the hypothesis in the statement of Theorem \ref{Teo4}.
\end{itemize}
From this we deduce that $T$ and $T^*_J$ remain dual in the universal cover of $Y$, and the same holds for their translated copies under the $D_\infty$-action. It follows that also $T$ and $T^*_J$ are dual in the universal cover of $Y_J$, and so we can find an isomorphism between the equivariant intersection forms of $Y$ and $Y_J$, sending the pair $(T,T^*)$ to $(T,T^*_J)$ (and the same for their images via the $D_\infty$ action) and leaving everything else unchanged.\\

By Theorem \ref{ThankyouMark}, the manifolds $Y_J$ will be homeomorphic to $Y$ if the Postnikov $k$-invariant is preserved after the surgeries.\\
Given any knot $J$, there exists a degree one map from $(S^3-J)$ to $(S^3-U)$ (where $U$ denotes the unknot) that collapses the Seifert surface of $J$ to a disk \cite{boileau2016one}. This map extends to a map $F$ from $Y_J$ to $Y$. If this map induces a homotopy equivalence between the associated Postnikov 2-approximations, then the $k$ invariant is preserved. To this end, we just need to check that $F$ induces an isomorphism between the first two homotopy groups of $Y_J$ and $Y$. This follows because:
\begin{itemize}
    \item we can find generators of the fundamental groups of $Y$ that are disjoint from the surgered torus, and
    \item by identifying $\pi_2(Y)$ with the second homology group of its universal cover, we see that $F$ maps $T^*_J$ to $T^*$ (and similarly for their images under the $D_\infty$ action) and leaves everything else unchanged, so $F$ induces an isomorphism between $\pi_2(Y_J)$ and $\pi_2(Y)$ (and of the respective equivariant intersection forms).  
\end{itemize}
 This concludes the proof of Theorem \ref{Teo4}.
\end{proof}

\section{Proof of Theorem 1.1}\label{Section5}
We will now prove Theorem \ref{Teo1}. We rephrase it here as follow, including also Theorem \ref{exoticS1}.

\begin{theorem}\label{TeoexoticWm}
There exist infinite families of irreducible, homeomorphic but pairwise not diffeomorphic non-spin manifolds $\{Y_{m,J}\}$, with $m\geq6$, and $\{Y'_{m,J}\}$, with $m\geq5$, with fundamental group isomorphic to $D_\infty$. Moreover, $b_2^+(Y_{m,J})=b_2^-(Y_{m,J})=2m+1$, while  $b_2^+(Y'_{m,J})=b_2^-(Y'_{m,J})=2m+2$.\\
 The double covers associated to $\Z\triangleleft D_\infty$ are, respectively, irreducible manifolds $\{W_{m,J}\}$ homeomorphic but pairwise not-diffeomorphic to $(S^1\times S^3) \# (4m+4) (\CP^2 \# \overline{\CP}^2 )$, and $\{W'_{m,J}\}$ homeomorphic but pairwise not-diffeomorphic to $(S^1\times S^3) \# (4m+6) (\CP^2 \# \overline{\CP}^2 )$.
\end{theorem}
\begin{proof}
    We use Theorem \ref{Teo4} (2) with input the Baykur-Hamada manifolds $Z_m$ described in Theorem \ref{sign0manifolds}.\\
    We can run the construction in two ways:
    \begin{enumerate}
        \item when $m\geq 6$, we pick disjointly embedded Lagrangian tori $T_1$, $T_2$, $T_3$ as in Corollary \ref{cor1}. We set $S_1=T_1$, $S_2=T_2$, and use $T_3$ to perform knot surgery.
        \item When $m\geq 5$ we use the genus $2$ surface  $\Sigma$ defined in Section \ref{subsectManifldZm} as $S_1$, we set $S_2=T_1$, and we use $T_2$ to perform knot surgery, where $T_1$ and $T_2$ are as in Corollary \ref{cor1}.
    \end{enumerate}
    All the surfaces described above have zero self-intersection and admit a dual surface. For example, all the tori have zero self-intersection because they are Lagrangian, and their dual is described in Section \ref{subsectManifldZm} (\ref{newtori}), while the dual of $\Sigma$ is given by the fiber of the fibration $\textgoth{Z}_m$. Moreover, we can perturb the symplectic structure so that the tori become symplectic.\\

    We verify that the hypotheses of Theorem \ref{Teo4} (2) are satisfied: by Corollary \ref{cor2}, the complement of the surfaces $S_1$ and $S_2$ in $Z_{m,J}$ is simply connected.  Moreover, by Lemma \ref{Pi1}, we have that $K\cup B\cup(C-\Sigma)$ is simply connected, and since $T_1$, $T_2$ and $T_2^*$ are contained in its complement, we deduce that every curve in $\Sigma$ and $T_1$ bounds a surfaces disjoint from $T_2$ and $T^*_2$ and intersects $\Sigma \cup T_1$ only along the boundary (and similarly, curves in $T_1$ and $T_2$ bound surfaces disjoint from $T_3$ and $T^*_3$). \\
Thus, we can apply Theorem \ref{Teo4}. In the first case we denote the exotic families as $Y_{m,J}$ and their double covers $W_{m,J}$, while in the second, we denote them $Y'_{m,J}$ and $W'_{m,J}$. Notice that the exoticness of $Y_{m,J}$ and $Y'_{m,J}$ is detected by the Seiberg-Witten invariants of $W_{m,J}$ and $W'_{m,J}$.\\



We can then easily compute the topological invariants of these manifolds. For $Y_m$: since removing tori and gluing along three-manifolds does not change the Euler characteristic, we first see that $\chi(W_m)=2\chi(Z_m)=2(4m+4)$ (recall that $Z_m$ is homeomorphic to $(2m+1)(\CP^2\# \overline{\CP}^2)$, as explained in Section \ref{subsectManifldZm}). Since $b_1(W_m)=1$ and $\sigma(W_m)=0$ (in fact, $W_m$ is cobordant to the disjoint union of two copies of $Z_m$), we obtain $b_2^+(W_m)=b_2^-(W_m)=4m+4.$\\
By multiplicativity of the Euler characteristic and of the signature, $\chi(Y_m)=\frac{1}{2}\chi(W_m)=\chi(Z_m)$  and $\sigma(Y_m)=\frac{1}{2}\sigma(W_m)=0$. Since a manifold with $D_\infty$ fundamental group has $b_1=0$, we find 
$$b_2^+(Y_m)=b_2^-(Y_m)=2m+1.$$

For $W_m'$ and $Y_m'$, the only difference is that, when removing a section, we increase the Euler characteristic by $2$. Thus, similar computations show that $b_2^+(W'_m)=b_2^-(W_m')=4m+6$ and
$$b_2^+(Y_m')=b_2^-(Y_m')=2m+2.$$

Since the condition $b_2-|\sigma|\geq 6$ holds in $W_{m,J}$ and $W_{m,J}'$, their equivariant intersection form is extended from the integers by (\cite{hambletonteich}, Theorem 2). By the homeomorphism classification of manifolds with infinite cyclic fundamental group \cite{freedman2014topology}, we deduce that they are homeomorphic to a connected sum of $S^1\times S^3$ and a simply-connected manifold with odd intersection form (in fact, by Section \ref{subsectTheFibrZm}(\ref{oddity}), we can find a surface with odd self-intersection in $W_{m,J}$), and this is enough to determine the homeomorphism type of the manifolds.\\

Lastly, we argue that these examples are irreducible. In fact, $Z_m$ is an irreducible symplectic manifold, as well as $Z_{m,J}$, if the knots used for knot surgeries are fibered, as seen in Section \ref{subhowgetexotica}. Thus, when gluing two copies of $Z_{m,J}$ by normal connected sum, we get a minimal symplectic manifold by  \cite{usher2006minimality}. Since it is also simply connected, it is irreducible by \cite{hamilton2006minimality}. Inductively, we see that gluing a finite number of copies of $Z_m$ by normal connected sum yields an irreducible manifold, as long as the result is simply connected.\\
So, an embedded $S^3$ in $Y_{m,J}$ or $W_{m,J}$ lifts to the universal cover $\Tilde{Y}_{m,J}$ described in Figure 
\ref{univcover}. By compactness, the lift is contained in a finite number of (glued) copies of $Z_m$, which we argued is irreducible. This concludes the proof.
\end{proof}
\begin{proof}[Proof of Theorem\ref{Teo1}]
Theorem \ref{TeoexoticWm} describes the signature zero examples of Theorem \ref{Teo1}. We get all the other examples by blowing up and changing the orientation. Their diffeomorphism type are distinguished by blowing-up and changing orientation on the manifolds $W_{m,J}$.
\end{proof}

\begin{section}{Manifolds with isometric equivariant intersection form}
    Here we briefly see how to apply Theorem \ref{Teo4} Case (1) to produce an infinite family of non-diffeomorphic $4$-manifolds, sharing the same equivariant intersection form. As previously explained, in this case, we are not able to say that the manifolds are homeomorphic; nevertheless, we have the following:
    \begin{theorem}
        There exist infinite families of pairwise not diffeomorphic non-spin manifolds $\{Y_{5,k}\}$ with $\pi_1(Y_{5,k})\cong D_\infty$, with isometric equivariant intersection form and $b_2^+(Y_{5,k})=b_2^-(Y_{5,k})=11$.
    \end{theorem}
\end{section}

\begin{proof}
    We apply Theorem \ref{Teo4} case (1) with input the Baykur-Hamada manifold $Z_5$, and with $S_1=T_1$ and $S_2=T_2$ as in Corollary \ref{cor1}, see also Figure \ref{fibration}. The null-homologous torus $T$ is defined in Definition \ref{def1}, and the manifold $(Z_5)_T(0)$ is symplectic, since $Z_5$ is obtained from it by Luttinger surgery. Also, $T_0\subset(Z_5)_T(0)$ admits a dual torus with zero self-intersection.
    Then, by Corollary \ref{cor2}, the hypotheses of Theorem \ref{Teo4} are satisfied, and thus we obtain a family $Y_{5,k}$ of non-diffeomorphic manifolds with the same equivariant intersection form. The same topological computations as in Section \ref{Section5} show that \[b_2^+(Y_{5,k})=b_2^-(Y_{5,k})=11.\qedhere\]
    \end{proof}
    The advantage of this construction is that we only need two tori with simply-connected complement. So, we do not need to enlarge the genus of the fiber of $Z_m$ to obtain the third one; compare with Section \ref{Section5}. Thus, we obtain smaller examples than in Theorem \ref{Teo1}.

   Of course, you can blow-up and change orientation to get families with $b_2^+ \neq b_2^-$.


\nocite{*}
\bibliographystyle{alpha}
\bibliography{ref}
\bigskip

\textit{E-mail address}: \texttt{simo@renyi.hu}

\bigskip

\textsc{$^{1}$E\"{o}tv\"{o}s Lor\'{a}nd University, Budapest, Hungary}\par
\textsc{$^{2}$HUN-REN Alfréd R\'{e}nyi Institute of Mathematics, Budapest, Hungary}

\end{document}

%% file: images/Fibration.pdf_tex
\begingroup%
  \makeatletter%
  \providecommand\color[2][]{%
    \errmessage{(Inkscape) Color is used for the text in Inkscape, but the package 'color.sty' is not loaded}%
    \renewcommand\color[2][]{}%
  }%
  \providecommand\transparent[1]{%
    \errmessage{(Inkscape) Transparency is used (non-zero) for the text in Inkscape, but the package 'transparent.sty' is not loaded}%
    \renewcommand\transparent[1]{}%
  }%
  \providecommand\rotatebox[2]{#2}%
  \newcommand*\fsize{\dimexpr\f@size pt\relax}%
  \newcommand*\lineheight[1]{\fontsize{\fsize}{#1\fsize}\selectfont}%
  \ifx\svgwidth\undefined%
    \setlength{\unitlength}{264.24987998bp}%
    \ifx\svgscale\undefined%
      \relax%
    \else%
      \setlength{\unitlength}{\unitlength * \real{\svgscale}}%
    \fi%
  \else%
    \setlength{\unitlength}{\svgwidth}%
  \fi%
  \global\let\svgwidth\undefined%
  \global\let\svgscale\undefined%
  \makeatother%
  \begin{picture}(1,0.58800471)%
    \lineheight{1}%
    \setlength\tabcolsep{0pt}%
    \put(0,0){\includegraphics[width=\unitlength,page=1]{Fibration.pdf}}%
    \put(0.57135574,0.52133838){\makebox(0,0)[lt]{\lineheight{1.25}\smash{\begin{tabular}[t]{l}$C_0$\end{tabular}}}}%
    \put(0.85890138,0.50046812){\makebox(0,0)[lt]{\lineheight{1.25}\smash{\begin{tabular}[t]{l}$\Sigma$\end{tabular}}}}%
    \put(0.56592119,0.26703023){\makebox(0,0)[lt]{\lineheight{1.25}\smash{\begin{tabular}[t]{l}$B_0$\end{tabular}}}}%
    \put(0.32190215,0.21383111){\makebox(0,0)[lt]{\lineheight{1.25}\smash{\begin{tabular}[t]{l}$K$\end{tabular}}}}%
    \put(-0.00251302,0.3033412){\makebox(0,0)[lt]{\lineheight{1.25}\smash{\begin{tabular}[t]{l}$T_1$\end{tabular}}}}%
    \put(0.02243995,0.23730598){\makebox(0,0)[lt]{\lineheight{1.25}\smash{\begin{tabular}[t]{l}$T_2$\end{tabular}}}}%
    \put(0.38760791,0.40758655){\color[rgb]{0.80784314,0.39215686,0}\makebox(0,0)[lt]{\lineheight{1.25}\smash{\begin{tabular}[t]{l}$A$\end{tabular}}}}%
    \put(0,0){\includegraphics[width=\unitlength,page=2]{Fibration.pdf}}%
  \end{picture}%
\endgroup%

%% file: images/torussym.pdf_tex
\begingroup%
  \makeatletter%
  \providecommand\color[2][]{%
    \errmessage{(Inkscape) Color is used for the text in Inkscape, but the package 'color.sty' is not loaded}%
    \renewcommand\color[2][]{}%
  }%
  \providecommand\transparent[1]{%
    \errmessage{(Inkscape) Transparency is used (non-zero) for the text in Inkscape, but the package 'transparent.sty' is not loaded}%
    \renewcommand\transparent[1]{}%
  }%
  \providecommand\rotatebox[2]{#2}%
  \newcommand*\fsize{\dimexpr\f@size pt\relax}%
  \newcommand*\lineheight[1]{\fontsize{\fsize}{#1\fsize}\selectfont}%
  \ifx\svgwidth\undefined%
    \setlength{\unitlength}{471.35363379bp}%
    \ifx\svgscale\undefined%
      \relax%
    \else%
      \setlength{\unitlength}{\unitlength * \real{\svgscale}}%
    \fi%
  \else%
    \setlength{\unitlength}{\svgwidth}%
  \fi%
  \global\let\svgwidth\undefined%
  \global\let\svgscale\undefined%
  \makeatother%
  \begin{picture}(1,0.79372327)%
    \lineheight{1}%
    \setlength\tabcolsep{0pt}%
    \put(0,0){\includegraphics[width=\unitlength,page=1]{torussym.pdf}}%
  \end{picture}%
\endgroup%

%% file: images/schematic.pdf_tex
\begingroup%
  \makeatletter%
  \providecommand\color[2][]{%
    \errmessage{(Inkscape) Color is used for the text in Inkscape, but the package 'color.sty' is not loaded}%
    \renewcommand\color[2][]{}%
  }%
  \providecommand\transparent[1]{%
    \errmessage{(Inkscape) Transparency is used (non-zero) for the text in Inkscape, but the package 'transparent.sty' is not loaded}%
    \renewcommand\transparent[1]{}%
  }%
  \providecommand\rotatebox[2]{#2}%
  \newcommand*\fsize{\dimexpr\f@size pt\relax}%
  \newcommand*\lineheight[1]{\fontsize{\fsize}{#1\fsize}\selectfont}%
  \ifx\svgwidth\undefined%
    \setlength{\unitlength}{350.6892425bp}%
    \ifx\svgscale\undefined%
      \relax%
    \else%
      \setlength{\unitlength}{\unitlength * \real{\svgscale}}%
    \fi%
  \else%
    \setlength{\unitlength}{\svgwidth}%
  \fi%
  \global\let\svgwidth\undefined%
  \global\let\svgscale\undefined%
  \makeatother%
  \begin{picture}(1,0.6272276)%
    \lineheight{1}%
    \setlength\tabcolsep{0pt}%
    \put(0,0){\includegraphics[width=\unitlength,page=1]{schematic.pdf}}%
    \put(0.04932912,0.51085899){\makebox(0,0)[lt]{\lineheight{1.25}\smash{\begin{tabular}[t]{l}$Z$\end{tabular}}}}%
    \put(0.29529629,0.4887491){\makebox(0,0)[lt]{\lineheight{1.25}\smash{\begin{tabular}[t]{l}$Z$\end{tabular}}}}%
    \put(0.45795807,0.49153404){\makebox(0,0)[lt]{\lineheight{1.25}\smash{\begin{tabular}[t]{l}$Z$\end{tabular}}}}%
    \put(0.73506156,0.55366946){\color[rgb]{0.91372549,0,0}\makebox(0,0)[lt]{\lineheight{1.25}\smash{\begin{tabular}[t]{l}$\nu S_1$\end{tabular}}}}%
    \put(0.43734563,0.57684326){\color[rgb]{0.91372549,0,0}\makebox(0,0)[lt]{\lineheight{1.25}\smash{\begin{tabular}[t]{l}$\nu S_1$\end{tabular}}}}%
    \put(0.88851243,0.55355923){\color[rgb]{0.93333333,0,0}\makebox(0,0)[lt]{\lineheight{1.25}\smash{\begin{tabular}[t]{l}$\nu S_1$\end{tabular}}}}%
    \put(0,0){\includegraphics[width=\unitlength,page=2]{schematic.pdf}}%
    \put(0.74283461,0.50179929){\color[rgb]{0,0,0.99215686}\makebox(0,0)[lt]{\lineheight{1.25}\smash{\begin{tabular}[t]{l}$\nu S_2$\end{tabular}}}}%
    \put(0.27238811,0.12587679){\color[rgb]{0,0,0.99215686}\makebox(0,0)[lt]{\lineheight{1.25}\smash{\begin{tabular}[t]{l}$\nu S_2$\end{tabular}}}}%
    \put(0.08271556,0.12577462){\color[rgb]{0,0,0.99215686}\makebox(0,0)[lt]{\lineheight{1.25}\smash{\begin{tabular}[t]{l}$\nu S_2$\end{tabular}}}}%
    \put(0.46448662,0.12654024){\color[rgb]{0,0,0.99215686}\makebox(0,0)[lt]{\lineheight{1.25}\smash{\begin{tabular}[t]{l}$\nu S_2$\end{tabular}}}}%
    \put(0.88861181,0.50284982){\color[rgb]{0,0,1}\makebox(0,0)[lt]{\lineheight{1.25}\smash{\begin{tabular}[t]{l}$\nu S_2$\end{tabular}}}}%
    \put(0.08407262,0.19082618){\color[rgb]{0.91372549,0,0}\makebox(0,0)[lt]{\lineheight{1.25}\smash{\begin{tabular}[t]{l}$\nu S_1$\end{tabular}}}}%
    \put(0.46186726,0.19329246){\color[rgb]{0.91372549,0,0}\makebox(0,0)[lt]{\lineheight{1.25}\smash{\begin{tabular}[t]{l}$\nu S_1$\end{tabular}}}}%
    \put(0.27606021,0.18819203){\color[rgb]{0.91372549,0,0}\makebox(0,0)[lt]{\lineheight{1.25}\smash{\begin{tabular}[t]{l}$\nu S_1$\end{tabular}}}}%
    \put(0.2618944,0.57949989){\color[rgb]{0.91372549,0,0}\makebox(0,0)[lt]{\lineheight{1.25}\smash{\begin{tabular}[t]{l}$\nu S_1$\end{tabular}}}}%
    \put(0.74951578,0.37278764){\makebox(0,0)[lt]{\lineheight{1.25}\smash{\begin{tabular}[t]{l}$W$\end{tabular}}}}%
    \put(0.83538019,0.19346194){\makebox(0,0)[lt]{\lineheight{1.25}\smash{\begin{tabular}[t]{l}$Y$\end{tabular}}}}%
    \put(0,0){\includegraphics[width=\unitlength,page=3]{schematic.pdf}}%
    \put(0.64725396,0.19831696){\makebox(0,0)[lt]{\lineheight{1.25}\smash{\begin{tabular}[t]{l}...\end{tabular}}}}%
    \put(0,0){\includegraphics[width=\unitlength,page=4]{schematic.pdf}}%
    \put(-0.0024394,0.11884526){\makebox(0,0)[lt]{\lineheight{1.25}\smash{\begin{tabular}[t]{l}...\end{tabular}}}}%
    \put(0.3788878,0.39409016){\makebox(0,0)[lt]{\lineheight{1.25}\smash{\begin{tabular}[t]{l}$X$\end{tabular}}}}%
    \put(0.32393453,0.00372037){\makebox(0,0)[lt]{\lineheight{1.25}\smash{\begin{tabular}[t]{l}$\Tilde{Y}$\end{tabular}}}}%
  \end{picture}%
\endgroup%

%% file: images/univcoverpicture2.pdf_tex
\begingroup%
  \makeatletter%
  \providecommand\color[2][]{%
    \errmessage{(Inkscape) Color is used for the text in Inkscape, but the package 'color.sty' is not loaded}%
    \renewcommand\color[2][]{}%
  }%
  \providecommand\transparent[1]{%
    \errmessage{(Inkscape) Transparency is used (non-zero) for the text in Inkscape, but the package 'transparent.sty' is not loaded}%
    \renewcommand\transparent[1]{}%
  }%
  \providecommand\rotatebox[2]{#2}%
  \newcommand*\fsize{\dimexpr\f@size pt\relax}%
  \newcommand*\lineheight[1]{\fontsize{\fsize}{#1\fsize}\selectfont}%
  \ifx\svgwidth\undefined%
    \setlength{\unitlength}{462.32052184bp}%
    \ifx\svgscale\undefined%
      \relax%
    \else%
      \setlength{\unitlength}{\unitlength * \real{\svgscale}}%
    \fi%
  \else%
    \setlength{\unitlength}{\svgwidth}%
  \fi%
  \global\let\svgwidth\undefined%
  \global\let\svgscale\undefined%
  \makeatother%
  \begin{picture}(1,0.44138523)%
    \lineheight{1}%
    \setlength\tabcolsep{0pt}%
    \put(0,0){\includegraphics[width=\unitlength,page=1]{univcoverpicture2.pdf}}%
    \put(0.41043654,0.24364529){\makebox(0,0)[lt]{\lineheight{1.25}\smash{\begin{tabular}[t]{l}$S_1$\end{tabular}}}}%
    \put(0.5436877,0.24510326){\makebox(0,0)[lt]{\lineheight{1.25}\smash{\begin{tabular}[t]{l}$S'_1$\end{tabular}}}}%
    \put(0.40818904,0.29261419){\makebox(0,0)[lt]{\lineheight{1.25}\smash{\begin{tabular}[t]{l}$S_2$\end{tabular}}}}%
    \put(0.54095624,0.29247725){\makebox(0,0)[lt]{\lineheight{1.25}\smash{\begin{tabular}[t]{l}$S'_2$\end{tabular}}}}%
    \put(0,0){\includegraphics[width=\unitlength,page=2]{univcoverpicture2.pdf}}%
    \put(0.52720258,0.02041993){\makebox(0,0)[lt]{\lineheight{1.25}\smash{\begin{tabular}[t]{l}$\tau$\end{tabular}}}}%
    \put(0.50334137,0.39114172){\makebox(0,0)[lt]{\lineheight{1.25}\smash{\begin{tabular}[t]{l}$\sigma$\end{tabular}}}}%
    \put(0.8533814,0.38557457){\makebox(0,0)[lt]{\lineheight{1.25}\smash{\begin{tabular}[t]{l}$\tau^2\sigma$\end{tabular}}}}%
    \put(0.1506171,0.39181764){\makebox(0,0)[lt]{\lineheight{1.25}\smash{\begin{tabular}[t]{l}$\tau^{-2}\sigma$\end{tabular}}}}%
    \put(0,0){\includegraphics[width=\unitlength,page=3]{univcoverpicture2.pdf}}%
    \put(0.68620402,0.37894495){\makebox(0,0)[lt]{\lineheight{1.25}\smash{\begin{tabular}[t]{l}$\tau\sigma$\end{tabular}}}}%
    \put(0.34111928,0.38052394){\makebox(0,0)[lt]{\lineheight{1.25}\smash{\begin{tabular}[t]{l}$\tau^{-1}\sigma$\end{tabular}}}}%
    \put(0,0){\includegraphics[width=\unitlength,page=4]{univcoverpicture2.pdf}}%
  \end{picture}%
\endgroup%